\documentclass{amsart}
\usepackage[utf8]{inputenc}
\usepackage{lmodern}
\usepackage{amsthm}
\usepackage{mathtools}
\usepackage{hyperref}
\usepackage{amssymb}
\usepackage{amsmath}
\usepackage{csquotes}
\usepackage{enumitem}
\usepackage{tikz-cd}
\usepackage{tikz,amsmath,amssymb,tkz-euclide}

\newtheorem{thrm}{Theorem}

\newtheorem{prop}[thrm]{Proposition}
\newtheorem{lemma}[thrm]{Lemma}

\theoremstyle{definition}

\theoremstyle{remark}

\numberwithin{equation}{section}

\numberwithin{equation}{section}

\usepackage{amssymb}

\title{Geometry of uniqueness varieties for a three-point Pick problem in $\mathbb{D}^3$}
\author{Krzysztof Maciaszek}
\date{April 13, 2022}

\begin{document}
	
	\maketitle
	\let\thefootnote\relax\footnote{The author was supported by the NCN Grant 2017/26/E/ST1/00723.}
	\begin{abstract}
		Motivated by the recent progress of research on extending holomorphic functions defined on subvarieties of classical domains and its connections to the 3-point Pick interpolation, we study a special class of two-dimensional algebraic subvarieties $M_\alpha$ of the unit tridisc, defined as the sets
		$$\lbrace (z_1,z_2,z_3)\in \mathbb{D}^3:\alpha_1z_1+\alpha_2z_2+\alpha_3z_3=\overline{\alpha}_1z_2z_3+\overline{\alpha}_2z_1z_3+\overline{\alpha}_3z_1z_2\rbrace.$$
		In this paper we show that given non-degenerated extremal maximal $3$-point Pick problem there exists an $\alpha$ such that $M_\alpha$ appears as its uniqueness variety. We also describe several geometric properties of $M_\alpha$ and show the biholomorphic equivalence between any two surfaces $M_\alpha$ and $M_\beta$, where the triples $\alpha$ and $\beta$ satisfy the so called triangle inequality.
	\end{abstract}
	
	\section{Introduction}
	Let $\Omega$ be a domain in $\mathbb{C}^n$ and $V\subset \Omega$ an arbitrary subset. Suppose $f:V\to \mathbb{C}$ is a function such that for any point $z\in V$ there exists its open neighbourhood $U\subset \Omega$ and there is a holomorphic function $F\in \mathcal{O}(U)$ such that $F$ coincide with $f$ on $V\cap U$. We call such a function $f$ holomorphic on $V$. By $H^\infty (V)$ we denote the algebra of bounded holomorphic functions on $V$.
	
	The set $V$ is said to have the (norm-preserving) extension property in $\Omega$, if $V\subset \Omega$ and for each function $f\in H^\infty(V)$, there exists $F\in H^\infty (\Omega)$ which admits 
	$$F|_V\equiv f \quad \text{ and } \quad ||F||_\Omega =||f||_V.$$
	
	The norm-preserving extensions of bounded holomorphic functions from subvarieties were studied (to our knowledge) for the first time by Rudin \cite{Rud}. However, the crucial for understanding the vague of the extension problem was the Agler and M$^c$Carthy's paper \cite{AglMc}, where the authors provided a complete solution for subvarieties of the bidisc. They're results gave motivation that links this problem to the Nevanlinna-Pick interpolation and to the Operator Theory. Moreover the tools developed and used (e.g. see \cite{Tho}) had a significant impact on further research. For more about motivations and some history we refer to \cite{AglKosMc} and \cite{KosMc2}.
	
	Although it has been almost two decades, the situation in higher dimensional polydiscs is only partially understood. The results that were particularly motivating to this work were obtained by Kosiński and M$^c$Carthy \cite{KosMc} in the case of tridisc $\mathbb{D}^3$, further developed by Kosiński and Zwonek \cite{KosZw}.
	
	In particular (in \cite{KosMc}) the authors have shown, for the two-dimensional polynomially convex analytic varieties, that extension property implies the existence of some subdomain $D\subset \mathbb{D}^2$ of the bidisc and a holomorphic function $f\in \mathcal{O}(D,\mathbb{D})$ such that $V$ can be written as
	\begin{equation}\label{eqf}
		\lbrace (z_1,z_2,f(z_1,z_2)):(z_1,z_2)\in D\rbrace .
	\end{equation}
	However it is still an open question concerning conditions being at the same time necessary and sufficient for $V$ to be a retract. It is equivalent to the question whether domain $D$ is the whole bidisc.
	
	Recall that a set $V$ is a (holomorphic) retract of $D$ if there exists a holomorphic function $r:D\to V$ which is the identity on $V$. 
	
	It is worth mentioning that not every set of the form (\ref{eqf}) has extension property and in general the question whether a particular (analytic) subset of $\mathbb{D}^3$ admits extension property is difficult. The authors in \cite{KosMc} took attention to the very interesting class of two-dimensional algebraic varieties, denote them by $M_\alpha$, which are not holomorphic retracts of the tridisc but it is still unknown whether they have the extension property for any $\alpha$. We define them as follows. For given $\alpha=(\alpha_1, \alpha_2, \alpha_3)\in \mathbb{C}^3\setminus \lbrace 0 \rbrace$ let $M_\alpha$ to be a set
	\begin{equation}
		\lbrace (z_1,z_2,z_3)\in \mathbb{D}^3:\alpha_1z_1+\alpha_2z_2+\alpha_3z_3=\overline{\alpha}_1z_2z_3+\overline{\alpha}_2z_1z_3+\overline{\alpha}_3z_1z_2\rbrace .
	\end{equation}
	Although the definition is correct for any $\alpha \in \mathbb{C}^3\setminus \lbrace 0 \rbrace$, only a special subfamily requires attention. Indeed, if the inequality $|\alpha_1|+|\alpha_2|\leq |\alpha_3|$ is satisfied (up to a permutation of coordinates), then $M_\alpha$ is biholomorphic to the bidisc $\mathbb{D}^2$ and then it is obviously a retract of the tridisc. Hence a class of varieties we shall consider (denote it $\mathcal{M}$), consists of those sets $M_\alpha$ for which the so called triangle inequality holds, i.e. $|\alpha_{i_1}|+|\alpha_{i_2}|>|\alpha_{i_3}|$ for all permutations of indices.
	
	First in-depth studies of the family $\mathcal{M}$ were carried out in \cite{KosZw}, where the authors described several properties of $M_\alpha$. In particular each such surface is a Carath\'eodory set in $\mathbb{D}^3$ i.e. the Carath\'eodory pseudodistance 
	\begin{equation*}
		c_{M_\alpha}(z,w):=\sup \lbrace \rho (F(z),F(w)):F\in \mathcal{O}(M_\alpha,\mathbb{D})\rbrace
	\end{equation*}
	coincide with the Carath\'eodory distance $c_{\mathbb{D}^3}(z,w)$ for all $z,w\in M_\alpha$. Here, $\rho$ is the hyperbolic metric on $\mathbb{D}$. Moreover it was noted that $M_\alpha$ is stable under automorphisms of the tridisc, which means that for any $\alpha$ satisfying the triangle inequality and any $m\in \operatorname{Aut}(\mathbb{D}^3)$ there is $\beta$ that also satisfies the triangle inequality and $m(M_\alpha)=M_\beta$. In Theorem \ref{thrm1} we show that in fact for any $\alpha$ and $\beta$ admitting the triangle inequality, the surfaces $M_\alpha$ and $M_\beta$ are biholomorphically equivalent, and desired biholomorphism can be settled by an automorphism of the tridisc. 
	
	Note that if $\alpha_3\neq 0$, then setting $a=\frac{\alpha_1}{\overline{\alpha_3}}, b=\frac{\alpha_2}{\overline{\alpha_3}}$ and $\omega=\frac{\overline{\alpha_3}}{\alpha_3}$, we can describe $M_\alpha$ as a graph of a function
	\begin{equation}\label{z3}
		z_3(z_1,z_2)=\omega \frac{az_1+bz_2-z_1z_2}{\overline{b}z_1+\overline{a}z_2-1}.
	\end{equation}
	
	The varieties $M_\alpha$ also have realization as two-dimensional domains
	\begin{equation}
		D_{a,b}=\lbrace (z_1,z_1)\in \mathbb{D}^2: |f_{a,b}(z_1,z_2)|<1\rbrace,
	\end{equation}
	where $a,b>0$, $f_{a,b}(z_1,z_2)=\frac{az_1+bz_2-z_1z_2}{bz_1+az_2-1}$ and the triple $(a,b,1)$ satisfy the triangle inequality.
	It is an interesting feature that the Lempert theorem holds for the domain $D_{a,b}$, although it is not even linearly convex. Hence it is natural to ask if $D_{a,b}$ is biholomorphic to a convex domain or if it can be exhausted by domains biholomorphic to convex domains.
	
	Another reason for studying the family $\mathcal{M}$ is that it coincide with the family of varieties, each of them being a unique solution to the maximal extremal 3-point Pick problem (see Theorem \ref{uniq}). Before we go further, let us recal some selected facts on the Pick interpolation problem in the tridisc $\mathbb{D}^3$.
	
	Set a pairwise distinct points $x,y,z\in \mathbb{D}^3$ and let $\alpha, \beta, \gamma \in \mathbb{D}$ be distinct. Consider a three-point Pick problem on the tridisc
	\begin{equation}\label{pick}
		\left\{\begin{matrix}
			x \mapsto \alpha,\\ 
			y \mapsto \beta, \\
			z \mapsto \gamma.
		\end{matrix}\right.
	\end{equation}
	The goal is to determine the existence of a function $F$ in the closed unit ball of $H^\infty (\mathbb{D}^n)$, satisfying $F(x)=\alpha, F(y)=\beta, F(z)=\gamma$. If such a function exists, it is called an interpolating function. Moreover, if there is no interpolating function with norm less than $1$, we say that the problem is extremal. The problem (\ref{pick}) is called non-degenerate if non of the two-point subproblems is extremal. 
	
	The extremal non-degenerated 3-point Pick problem was solved by Kosiński in \cite{Kos}. In particular, the author proved that all interpolating functions coincide on two-dimensional analytic variety of the special form. In Theorem \ref{uniq} (and preceding discussion) we will show this variety to be precisely one of the members of the family $\mathcal{M}$ and even more, that this is a maximal set where all interpolating functions coincide.
	
	\section{Results}
	Now we shall prove the first main result
	\begin{thrm}\label{thrm1}
		Let $\alpha=(\alpha_1,\alpha_2,\alpha_3)$ and $\beta=(\beta_1,\beta_2,\beta_3)$ both satisfies the triangle inequality. Then there exists $\varphi \in \operatorname{Aut}(\mathbb{D}^3)$ such that $\varphi(M_\alpha)=M_\beta$.
	\end{thrm}
	\begin{proof}
		It is enough to show that $M_{(1,1,1)}$ is biholomorphically equivalent to any $M_{\alpha}$, where the triple $\alpha$ does satisfy the triangle inequality. Assume $m(z_1,z_2,z_3):=(m_{\lambda}(z_1), m_{\mu}(z_2), m_{\gamma}(z_3))$ is the automorphism of the tridisc $\mathbb{D}^3$, consisting of M\"obius maps, where $\lambda, \mu, \gamma$ are real numbers and the point $(\lambda, \mu,\gamma)$ belongs to $M_{(1,1,1)}$. In particular $\gamma=\frac{\lambda+\mu-\lambda\mu}{\lambda+\mu-1}$. Therefore, for fixed such a triple the point $\left (\lambda, \mu, \frac{\lambda+\mu - \lambda \mu}{\lambda+\mu-1}\right )\in M_{(1,1,1)}$ is sent to the origin and the variety $M_{(1,1,1)}$ is mapped to the set of points satisfying the following equality
		\begin{multline*}
			\frac{\lambda-z_1}{1-\lambda z_1}+\frac{\mu-z_2}{1- \mu z_2}+\frac{\gamma-z_3}{1-\gamma z_3}
			\\
			=\frac{\lambda-z_1}{1-\lambda z_1}\frac{\mu-z_2}{1-\mu z_2}
			+\frac{\lambda-z_1}{1-\lambda z_1}\frac{\gamma-z_3}{1-\gamma z_3}
			+\frac{\mu-z_2}{1-\mu z_2}\frac{\gamma-z_3}{1-\gamma z_3}.
		\end{multline*}
		Simple calculations lead to the following equation:
		\begin{multline*}
			A(\lambda,\mu)z_1+B(\lambda,\mu)z_2+C(\lambda,\mu)z_3
			\\=A(\lambda,\mu)z_2z_3+B(\lambda, \mu)z_1z_3+C(\lambda, \mu)z_1z_2.
		\end{multline*}
		Here
		\begin{align*}
			A(\lambda, \mu)=(1-\lambda^2)(\mu+\gamma-1),
			\\
			B(\lambda, \mu)=(1-\mu^2)(\lambda+\gamma-1),
			\\
			C(\lambda, \mu)=(1-\gamma^2)(\lambda+\mu-1).
		\end{align*}
		Substitute the explicit formula for $\gamma=\gamma (\lambda, \mu)$ taken from the assumption, so we get
		\begin{align*}
			A(\lambda,\mu)&=\frac{(1-\lambda^2)(\mu^2-\mu+1)}{\lambda+\mu-1},
			\\
			B(\lambda,\mu)&=\frac{(1-\mu^2)(\lambda^2-\lambda+1)}{\lambda+\mu-1},
			\\
			C(\lambda,\mu)&=\frac{1+2(\lambda+\mu)(\lambda \mu-1)-\lambda^2\mu^2}{\lambda+\mu-1}.
		\end{align*}
		Set $\mathcal{A}(\lambda,\mu)=\frac{(1-\lambda^2)(\mu^2-\mu+1)}{1+2(\lambda+\mu)(\lambda \mu-1)-\lambda^2\mu^2}$ and $\mathcal{B}(\lambda,\mu)=\frac{(1-\mu^2)(\lambda^2-\lambda+1)}{1+2(\lambda+\mu)(\lambda \mu-1)-\lambda^2\mu^2}$. 
		Both functions are well defined on the set 
		$$V=\left \lbrace (\lambda,\mu)\in (-1,1)^2:\left | \frac{\lambda+\mu-\lambda \mu}{\lambda+\mu-1}\right |<1 \right \rbrace.$$ 
		Consider a set 
		\begin{equation*}
			\mathcal{W} := \lbrace (x_1,x_2)\in \mathbb{R}^2: |x_1|+|x_2|>1, |x_1|+1>|x_2|, |x_2|+1>|x_1|\rbrace.
		\end{equation*} 
		We shall show that the subset of $\mathcal{W}$, defined as 
		$$D=\left \lbrace (\mathcal{A}(\lambda,\mu), \mathcal{B}(\lambda,\mu)):(\lambda,\mu)\in V\right \rbrace ,$$ 
		is open-closed in connected component of $\mathcal{W}$. Moreover, the mapping that defines $D$ is continuous on $V$ and, since $(\mathcal{A}(0,0),\mathcal{B}(0,0))$ belongs to a connected component $\mathcal{W}_+=\mathcal{W} \cap \lbrace x_1>0, x_2>0\rbrace$. From this it will follow that $D=\mathcal{W}_+$.
		
		To prove that $D$ is open let take a map $\Psi:(\lambda, \mu)\mapsto (\mathcal{A}(\lambda,\mu),\mathcal{B}(\lambda,\mu))$. We will show it is a local diffeomorphism. Instead of $\Psi$, we can consider 
		$$\psi:(\lambda, \mu)\mapsto (\mathcal{A}(\lambda,\mu)+\mathcal{B}(\lambda,\mu),\mathcal{A}(\lambda,\mu)-\mathcal{B}(\lambda,\mu)).$$
		
		Note, that $(\mathcal{A}(\lambda,\mu)+\mathcal{B}(\lambda,\mu),\mathcal{A}(\lambda,\mu)-\mathcal{B}(\lambda,\mu))=\left ( \frac{\lambda+\mu-2\lambda \mu -2}{2\lambda+2\mu-\lambda \mu-1}, \frac{\lambda-\mu}{1-\lambda \mu}\right )$. The Jacobian matrix
		\begin{equation*}
			\begin{bmatrix}
				\frac{3(1-\mu^2)}{(2\lambda+2\mu-\lambda \mu-1)^2} & \frac{3(1-\lambda^2)}{(2\lambda+2\mu-\lambda \mu-1)^2}\\
				\frac{1-\mu^2}{(1-\lambda \mu)^2} & \frac{-1+\lambda^2}{(1-\lambda \mu )^2}
			\end{bmatrix}
		\end{equation*}
		has non-zero determinant at each point of $V$. From this fact it immediately follows that $D$ is open.
		
		Now we will show $D$ is also closed. We shall consider several cases.
		\\
		Suppose $(\lambda, \mu)$ goes to $(-1^+,-1^+)$. Then $\mathcal{A}(\lambda,\mu)+\mathcal{B}(\lambda,\mu)$ converge to $1$ and hence $(\mathcal{A}(\lambda,\mu),\mathcal{B}(\lambda,\mu))$ in the limit approaches a boundary point of $\partial \mathcal{W}$.
		\\
		Let now $\lambda \to 1^-$ and $\mu \to -1^+$. Then the difference $\mathcal{A}(\lambda,\mu)-\mathcal{B}(\lambda,\mu)$ goes to $1$. Similarly, it converge to $-1$ provided that $\lambda \to -1^+$ and $\mu \to 1^-$. In both cases we would reach the boundary point of $\mathcal{W}$.
		\\
		Note that the boundary of $V$ is a set 
		$$\partial V=\left \lbrace (\lambda,\mu)\in (0,1)^2: \mu = \frac{2\lambda - 1}{\lambda-2}\right \rbrace \cup \lbrace -1 \rbrace \times [-1,1] \cup [-1,1]\times \lbrace -1 \rbrace.$$
		
		Assuming $\lambda \to \lambda_0\in (-1,1)$ and $\mu \to \frac{2\lambda-1}{\lambda-2}$, we easily see that $\mathcal{A}(\lambda,\mu)+\mathcal{B}(\lambda,\mu)$ blows up to infinity. If $\lambda$ goes to $-1^+$ and $\mu \to (-1,1)$ then this sum converge to $1$. By symmetry we have the same result after permuting $\lambda$ with $\mu$.
		
		In each case we have obtained that $(\mathcal{A}(\lambda,\mu),\mathcal{B}(\lambda,\mu))$ goes to the boundary of $\mathcal{W}_+$ whenever $(\lambda,\mu)$ converge to the boundary of $V$.
		
		In order to finish the proof we will show that for any $\alpha=(\alpha_1,\alpha_2,\alpha_3)\in \mathbb{C}^3$ satisfying the triangle inequality there exists an automorphism $\varphi \in \operatorname{Aut}(\mathbb{D}^3)$ such that $\varphi(M_{\alpha})=M_{|\alpha|}$. Compose $M_\alpha$ with the rotations of the coordinates
		$$\mathbb{D}^3\ni (z_1,z_2,z_3)\overset{\phi}{\mapsto} \left (\omega_1 z_1,\omega_2 z_2,\omega_3 z_3\right ).$$
		Here $\omega_j$ are in the unit circle $\mathbb{T}$ for $j=1,2,3$. 
		The variety $\phi^{-1}(M_\alpha)$ can be rewritten as the set of points $(z_1,z_2,z_3)\in \mathbb{D}^3$ that satisfy
		\begin{multline*}
			\gamma \alpha_1\omega_1z_1+\gamma \alpha_2 \omega_2z_2+\gamma \alpha_3\omega_3z_3
			\\
			=\gamma \overline{\alpha}_1\omega_2\omega_3z_2z_3+\gamma \overline{\alpha}_2\omega_1\zeta_3z_1z_3+\gamma \overline{\alpha}_3\omega_1\omega_2z_1z_2.
		\end{multline*}
		We can see that $\phi^{-1}(M_\alpha)=M_\beta$ for some triple $\beta=(\beta_1,\beta_2,\beta_3)$. Indeed, let $\gamma^2=\frac{1}{\omega_1\omega_2\omega_3}$, then $\beta_j=\frac{\eta_j}{\eta_k\eta_l}a_j$, whenever $j\neq k \neq l$ and $\eta_j$ being any square root of $\omega_j$. Write $\alpha_j=|\alpha_j|e^{is_j}$ and $\eta_j=e^{it_j}$. In order to obtain $\beta=|\alpha|$ it is enough to choose $t=(t_1,t_2,t_3)$ to be a solution of
		$$At^T=s^T,$$
		where  
		\begin{equation*}
			A=\begin{bmatrix}
				1 & -1 & -1\\ 
				-1 & 1 & -1\\ 
				-1 & -1 & 1
			\end{bmatrix},  \quad s=(s_1,s_2,s_3).
		\end{equation*}
		The equivalence between any two $M_{\alpha}$ and $M_{\beta}$ we can now see on the diagram
		\[
		\begin{tikzcd}
			M_{\alpha} \arrow{r}{\varphi_1} \arrow[swap]{ddrr}{\varphi_4\circ ... \circ \varphi_1} & M_{|\alpha|} \arrow{r}{\varphi_2} & M_{(1,1,1)} \arrow{d}{\varphi_3} \\
			& & M_{|\beta|} \arrow{d}{\varphi_4}\\
			& & M_{\beta},
		\end{tikzcd}
		\] 
		where $\varphi_j$ for $j=1,...,4$ are appropriate biholomorphisms constructed as above.
		
		This finishes the proof.
	\end{proof}
	
	Consider a non-degenerate maximal extremal 3-point Pick problem. Because the group of automorphisms acts transitively on $\mathbb{D}^3$, we may consider the following situation
	\begin{equation}\label{NP}
		\left \{\begin{matrix}
			0 \mapsto 0;\\
			x \mapsto \alpha;\\
			y \mapsto \beta.
		\end{matrix} \right.
	\end{equation}
	
	For the next result we need some preparation. We will use the so called magic functions. The notion was formally introduced in \cite{AglYou} to study the geometry of domains (automorphisms, Carath\'eodory problem). These functions also play crucial role in the construction of the solution of 3-point Pick problem (cf. \cite{Kos}). In the bidisc we define magic functions as
	$$\Phi_{a,\eta}(x,y)=\frac{ax+(1-a)y+\eta xy}{1+\eta (1-a)x+\eta ay},$$
	where $a\in [0,1]$ and $\eta \in \mathbb{T}$. Kosiński proved \cite{Kos} that (\ref{NP}) can be interpolated by a function $F(z_1,z_2,z_3)=G_1(G_2(z_1,z_2),z_3)$ with $z\in \mathbb{D}^3$, where $G_1,G_2$ are magic functions. Moreover, there exist points $a_1, a_2, a_3 \in \mathbb{D}$, which are not co-linear, and such that an analytic disc
	$$\lambda \mapsto (\lambda m_{a_1} (\lambda), \lambda m_{a_2}(\lambda), \lambda m_{a_3} (\lambda))$$
	passes through the nodes. There is also a point $\gamma \in \mathbb{D}$ being a convex combination of $a_1, a_2, a_3$ for which 
	\begin{equation}\label{eq:Fa} F(\lambda m_{a_1} (\lambda), \lambda m_{a_2}(\lambda), \lambda m_{a_3} (\lambda))=\lambda m_{\gamma}(\lambda) \text{ for } \lambda \in \mathbb{D}.
	\end{equation}
	Therefore the Pick problem can be written as 
	\begin{equation}\label{NP2}
		\left\{\begin{matrix}
			0\mapsto 0;\\ 
			(xm_{a_1}(x),x m_{a_2}(x),x m_{a_3}(x))\mapsto xm_{\gamma}(x);\\
			(ym_{a_1}(y),y m_{a_2}(y),y m_{a_3}(y))\mapsto ym_{\gamma}(y).
		\end{matrix}\right.
	\end{equation}
	It was also shown in \cite{Kos} that \eqref{eq:Fa} and the Schwarz lemma implies
	that 
	\begin{equation}\label{eq:FA}
		F(\lambda m_{ta_1} (\lambda), \lambda m_{ta_2}(\lambda), \lambda m_{ta_3} (\lambda))=\lambda m_{t\gamma}(\lambda) \text{ for } \lambda \in \mathbb{D},\ t\in [0,1].
	\end{equation}
	This in particular (see also \cite{KosZw}) shows the following:
	\begin{lemma}\label{lem:Fa} The uniqueness variety (the set of points where all the interpolating functions coincide) contains a 3-real dimensional surface 
		$$\mathcal N_{(a_1, a_2, a_3)}:=\lbrace (\lambda m_{ta_1} (\lambda),\lambda m_{ta_2} (\lambda),\lambda m_{ta_3} (\lambda)):\lambda \in \mathbb{D}, t\in (0,1)\rbrace.$$
	\end{lemma}
	\begin{lemma}\label{lem:FA}  $\mathcal N_{(a_1, a_2, a_3)}$ is a subset of $M_\alpha$ for properly chosen $\alpha$.
	\end{lemma} 
	\begin{proof}To see it let
		$z_1=\lambda m_{a_1}(\lambda)$ and $z_2=\lambda m_{a_2}(\lambda)$. Then 
		$$\lambda t = \frac{z_1-z_2}{a_1-a_2+\overline{a_1}z_1-\overline{a_2}z_2}, \quad \lambda^2=\frac{a_2 z_1-a_1 z_2+(\overline{a_2}-\overline{a_1})z_1z_2}{a_1-a_2+\overline{a_1}z_1-\overline{a_2}z_2}.$$
		Now substitute these equalities in the third coordinate, so we get
		$$\lambda m_{a_3}(\lambda)=\omega \frac{\alpha z_1+\beta z_2- z_1z_2}{\overline{\beta}z_1+\overline{\alpha}z_2-1},$$
		where $\omega = -\frac{\overline{a_2}-\overline{a_1}}{a_2-a_1}$, $\alpha = \frac{a_3-a_2}{\overline{a_2}-\overline{a_1}}$ and $\beta = \frac{a_1-a_3}{\overline{a_2}-\overline{a_1}}$. This is exactly the form (\ref{z3}).
		
	\end{proof}
	
	Recall, for the situation in the bidisc it was previously known (cf. \cite{AglMc2}) that the uniqueness variety for $N$-point extremal Pick problem is either the whole bidisc or it contains at least a distinguished variety through the nodes. If $N=3$ it was shown in \cite{AglMc3} that if the uniqueness variety is not the whole bidisc, then the minimal extremal problem admits a solution being a function depending only of one coordinate.
	
	In the next result we show that every uniqueness variety for the maximal extremal non-degenerated 3-point Pick problem in $\mathbb{D}^3$ is a member of family $\mathcal{M}$. 
	
	For given data (\ref{NP2}) we can represent $\gamma$ by writing $\gamma=tb+(1-t)a_3$, where $b=sa_1+(1-s)a_2$. Then we construct (cf. \cite{Kos}) an interpolating function as $F_1(z_1,z_2,z_3)=\Phi_{t,\omega}(\Phi_{s,\nu}(z_1,z_2),z_3)$, with $\nu=\frac{\overline{a_2}-\overline{a_1}}{a_2-a_1}$ and $\omega=\frac{\overline{a_3}-\overline{b}}{a_3-b}$.
	
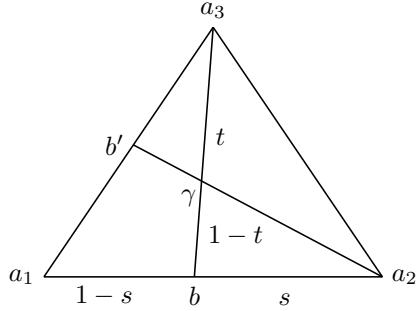
	\begin{center}
		\begin{tikzpicture}
			\tkzDefPoint(0,0){a_1}
			\tkzLabelPoints[below,left](a_1)
			\tkzDefPoint(2.25,3.320718914){a_3}
			\tkzLabelPoints[above](a_3)
			\tkzDefPoint(4.5,0){a_2}
			\tkzLabelPoints[below,right](a_2)
			\tkzDrawSegments(a_1,a_3 a_3,a_2 a_1,a_2)
			
			\tkzDefPoint(2,0){a}
			\tkzInterLL(a_1,a_2)(a_3,a) \tkzGetPoint{b}
			\tkzLabelPoints[below](b)
			\tkzDrawSegment(a_3,b)
			
			\tkzDefLine[bisector](a_3,a_2,a_1)\tkzGetPoint{a'}
			\tkzInterLL(a_3,a_1)(a_2,a') \tkzGetPoint{b'}
			\tkzLabelPoints[left](b')
			\tkzDrawSegment(a_2,b')
			
			\tkzDefPoint(2.15,1.05){g}
			\tkzLabelPoint[below,left](g){$\gamma$}
			
			\tkzDefPoint(0.8,0){s1}
			\tkzLabelPoint[below](s1){$1-s$}
			\tkzDefPoint(3.2,-0.1){s2}
			\tkzLabelPoint[below](s2){$s$}
			\tkzDefPoint(2.55,0.8){t1}
			\tkzLabelPoint[below](t1){$1-t$}
			\tkzDefPoint(2.35,2.1){t2}
			\tkzLabelPoint[below](t2){$t$}
		\end{tikzpicture}
	\end{center}
	
	Similarly, for the same data we can write $\gamma=t'b'+(1-t')a_2$ where $b'=s'a_1+(1-s')a_3$. Set $\nu'=\frac{\overline{a_3}-\overline{a_1}}{a_3-a_1}$ and $\omega=\frac{\overline{a_2}-\overline{b'}}{a_2-b'}$. The corresponding interpolating function is then $F_2(z_1,z_2,z_3)=\Phi_{t',\omega'}(\Phi_{s',\nu'}(z_1,z_3),z_2)$.
	
	\begin{thrm}\label{uniq}
		The family $\mathcal{M}$ coincide with family of uniqueness varieties for extremal, non-degenerated, strictly 3-dimensional three-point Pick problems in $\mathbb{D}^3$. 
	\end{thrm} 
	
	\begin{proof}
		According to \cite{Kos} any non-degenerate, strictly 3-dimensional three point Pick interpolation problem can be written, up to an automorphisms, as
		\begin{equation*}
			\left\{\begin{matrix}
				(xm_{a_1}(x),x m_{a_2}(x),x m_{a_3}(x))\mapsto xm_{\gamma}(x);\\
				(ym_{a_1}(y),y m_{a_2}(y),y m_{a_3}(y))\mapsto ym_{\gamma}(y);\\
				(zm_{a_1}(z),z m_{a_2}(z),z m_{a_3}(z))\mapsto zm_{\gamma}(z).
			\end{matrix}\right.
		\end{equation*}
		We can make some reductions that will simplify the computations we want to carry out. We can choose automorphism of the tridisc $\psi=(\psi_1,\psi_2,\psi_3)$ and an automorphism of the unit disc $\varphi$ such that for $\lambda \in \mathbb{D}$ we have
		\begin{equation*}
			\psi_j(\varphi(\lambda) m_{a_j}(\varphi (\lambda)))=\lambda m_{b_j}(\lambda)
		\end{equation*} for some $b_j$ such that $b_1=0$. Multiplying the variables and coordinates out by a unimodular constants we can assume that $b_2> 0$. 
		Composing interpolation function with a M\"obius map leads to an equivalent problem
		\begin{equation}\label{NP3}
			\left\{\begin{matrix}
				(-x^2,x m_{b_2}(x),x m_{b_3}(x))\mapsto xm_{\delta}(x);\\
				(-y^2,y m_{b_2}(y),y m_{b_3}(y))\mapsto ym_{\delta}(y);\\
				(-z^2,z m_{b_2}(z),z m_{b_3}(z))\mapsto zm_{\delta}(z).
			\end{matrix}\right.
		\end{equation}
		Denote by $V$ the uniqueness variety for (\ref{NP3}). Set 
		$$\Omega:=\lbrace (z_1,z_2,z_3) \in \mathbb{D}^3:F_1(z_1,z_2,z_3) = F_2(z_1,z_2,z_3)\rbrace,$$ 
		where $F_1, F_2$ are interpolating functions described as in the discussion preceding theorem, i.e. there are $s,t\in [0,1]$ and $\omega, \nu \in \mathbb{T}$ such that $F_1(z_1,z_2,z_3)=\Phi_{t,\omega} (\Phi_{s,\nu}(z_1,z_2),z_3)$. Similarly, for properly chosen $t',s'\in [0,1]$ and $\omega',\nu'\in \mathbb{T}$, we can write $F_2(z_1,z_2,z_3)=\Phi_{t',\omega'}(z_1,\Phi_{s',\nu'}(z_2,z_3))$.  
		
		Obviously $V\subset \Omega$. We already learned in Lemmas~\ref{lem:Fa} and \ref{lem:FA} that there exists $\alpha$ such that $M_\alpha \subset V$. We shall show that that 
		\begin{equation}\label{eq:OM}\Omega = M_\alpha,
		\end{equation} which clearly gives a desired equality: $M_\alpha= V$. 
		
		The proof of \eqref{eq:OM} boils down to direct computations. It follows from \cite{Kos} that $F_j(-\lambda^2, \lambda m_{tb_2}(\lambda), \lambda m_{tb_3}(\lambda)) = \lambda m_{t\delta}(\lambda)$ for $\lambda \in \mathbb{D}$ and $t\in (0,1)$. Hence we we lose no generality assuming that $b_2=1$. Then we have 
		\begin{align*}
			\nu&=1, \qquad \omega=\frac{1-s-\overline{b_3}}{1-s-b_3},\\
			\nu'&=\frac{\overline{b_3}}{b_3}, \qquad \omega'=\frac{(1-s')\overline{b_3}-1}{(1-s')b_3-1}.
		\end{align*}
		We define $F_1$ and $F_2$ as before:
		\begin{align*}
			F_1(z_1,z_2,z_3)&=\Phi_{t,\omega}(\Phi_{s,\nu}(z_1,z_2),z_3)
			\\
			&=\Phi_{t,\omega}\left (\frac{sz_1+(1-s)z_2+\nu z_1z_2}{1+\nu(1-s)z_1+\nu sz_2},z_3\right )
			\\
			&=\frac{t\frac{sz_1+(1-s)z_2+\nu z_1z_2}{1+\nu(1-s)z_1+\nu sz_2}+(1-t)z_3+\omega \frac{sz_1+(1-s)z_2+\nu z_1z_2}{1+\nu(1-s)z_1+\nu sz_2}z_3}{1+\omega (1-t)\frac{sz_1+(1-s)z_2+\nu z_1z_2}{1+\nu(1-s)z_1+\nu sz_2}+\omega t z_3}
			\\
			&=\frac{Az_1+Bz_2+Cz_3+Dz_1z_2+Ez_1z_3+Fz_2z_3 +\omega \nu z_1z_2z_3}{1+\omega \nu (\overline{F}z_1+ \overline{E}z_2+ \overline{D} z_3 + Cz_1z_2+ Bz_1z_3+ A z_2z_3)},
		\end{align*}
		where $A=ts$, $B=t(1-s)$, $C=(1-t)$, $D=\nu t$, $E=\omega s+\nu(1-s)(1-t)$ and $F=\nu s(1-t)+\omega (1-s)$. Similarly
		\begin{align*}
			F_2(z_1,z_2,z_3)&=\Phi_{t',\omega'}(\Phi_{s',\nu'}(z_1,z_3),z_2)
			\\
			&=\frac{A'z_1+B'z_2+C'z_3+D'z_1z_2+E'z_1z_3+F'z_2z_3 +\omega \nu z_1z_2z_3}{1+\omega' \nu' (\overline{F'}z_1+ \overline{E'}z_2+ \overline{D'} z_3 + C'z_1z_2+ B'z_1z_3+ A' z_2z_3)},
		\end{align*}
		with $A'=t's'$, $B'=(1-t')$, $C'=t'(1-s')$, $D'=\omega' s'+\nu'(1-s')(1-t')$, $E'=\nu' t'$ and $F'=\nu' s'(1-t')+\omega' (1-s')$.
		
		Consider the equation $F_1-F_2=0$. Some computations lead to
		\begin{equation}\label{eqdelta}
			(\alpha \beta'-\alpha'\beta)z_3^2+(\alpha \gamma'-\alpha'\gamma+\beta' \delta-\beta\delta')z_3+(\gamma' \delta-\gamma\delta')=0,
		\end{equation}
		where
		\begin{align*}
			\alpha=\alpha(z_1,z_2)&=C+Ez_1+Fz_2+\omega \nu z_1z_2,\\
			\beta= \beta(z_1,z_2)&=\omega\nu(\overline{D}+Bz_1+Az_2),
			\\
			\gamma=\gamma(z_1,z_2)&=1+\omega\nu(\overline{F}z_1+\overline{E}z_2+Cz_1z_2),\\
			\delta=\delta(z_1,z_2)&=Az_1+Bz_2+Dz_1z_2,\\
			\alpha'=\alpha'(z_1,z_2)&=C'+E'z_1+F'z_2+\omega' \nu' z_1z_2,\\
			\beta'= \beta'(z_1,z_2)&=\omega'\nu'(\overline{D'}+B'z_1+A'z_2),\\
			\gamma'=\gamma'(z_1,z_2)&=1+\omega'\nu'(\overline{F'}z_1+\overline{E'}z_2+C'z_1z_2),\\
			\delta'=\delta'(z_1,z_2)&=A'z_1+B'z_2+D'z_1z_2.
		\end{align*}
		As mentioned above, the uniqueness variety for (\ref{NP3}) contains the surface $$M_\alpha=\{(z_1, z_2, z_3(z_1, z_2))\},$$ where $z_3$ is a rational function of the form (\ref{z3}). Thus $z_3=z_3(z_1, z_2)$ is one of the solutions of equation (\ref{eqdelta}). Elementary but tedious computations (we carried them out in Mathematica) shows that discriminant of this quadratic polynomial is in fact equal to $0$. This shows that $\Omega$ coincides with $M_\alpha$.
	\end{proof} 
	
	In this part we will describe the Shilov boundary of $M_\alpha$. Let us first recall the admired Shilov result. Suppose $K$ is a compact space, $\mathcal{C}(K)$ an algebra of continuous functions on $K$ and $\mathcal{A}(K)$ is any subalgebra of $\mathcal{C} (K)$ that separates points. We say that a subset $\Gamma$ of $K$ is a boundary for $\mathcal{A}(K)$ if for each $f\in \mathcal{A}(K)$ we can find $w\in \Gamma$ so that
	$$|f(w)|=\max_{x\in K} |f(x)|.$$
	The Shilov Theorem \cite{Shi} says that the intersection of all boundaries is non-empty and it is again a boundary for $\mathcal{A}(K)$. We call it a Shilov boundary.
	
	\begin{prop}
		For any $\alpha$ satisfying the triangle inequality, the Shilov boundary of the surface $M_\alpha$ with respect to the algebra $\mathcal{O}(M_\alpha)\cap \mathcal{C}(\overline{M}_\alpha)$ is equal to $\overline{M}_\alpha \cap \mathbb{T}^3$. 
	\end{prop}
	\begin{proof}
		
		Without loss of generality we can assume that $\alpha = (1,1,1)$. It is obvious that $ \overline{M}_\alpha \cap \mathbb{T}^3\subset \partial_SM_\alpha$, where $\partial_SM_\alpha$ stands for the Shilov boundary of $M_\alpha$. Note that this set is non-empty for any choise of $\alpha$. We shall show that in $\overline{M}_\alpha$ there is no more peak points.
		
		Suppose there is a point $(\zeta_1,\zeta_2,\zeta_3) \in \partial_S M_\alpha$ which does not belong to the three-torus $\mathbb{T}^3$. We can assume (up to a permutation of coordinates) that $|\zeta_1|< 1$ and $|\zeta_2|=1$. Then the condition 
		$|\zeta_1+\zeta_2-\zeta_1 \zeta_2|^2\leq |\zeta_1+\zeta_2-1|^2$ is satisfied precisely when $2\operatorname{Re}(\zeta_2)(1-|\zeta_1|^2)\leq 1-|\zeta_1|^2.$ Hence $\operatorname{Re}(\zeta_2)\leq \frac{1}{2}$. We shall show that the following analytic disc $$\mathbb D\ni z \mapsto \xi(z,\zeta_2):=\left (z,\zeta_2,\frac{z+\zeta_2-z\zeta_2}{z+\zeta_2-1}\right)$$ belongs to the boundary $\partial M_\alpha$, which will clearly give a contradiction.
		
		To prove our claim take any sequence $w_n\in \mathbb{D}$ such that $\operatorname{Re}(w_n)<\frac{1}{2}$ and $w_n\to \zeta_2$ as $n$ goes to infinity. Then, for $n$ large enough, a strict inequality
		$$2\operatorname{Re}(z)(1-|w_n|^2)+2\operatorname{Re}(w_n)(1-|z|^2)<1-|z|^2|w_n|^2$$
		is satisfied. This means, that for such $n$ an inequality $\left |\frac{z+w_n-zw_n}{z+w_n-1}\right |<1$ is also true. Since the boundary is closed, the point $\xi (z,\zeta_2)$ is in $\partial M_\alpha$ and the proof is finished.
	\end{proof}
	
	\textbf{Acknowledgments.} The author would like to express his gratitude to the anonymous referee for his or her careful review and insightful comments that improved the shape of the paper.

 \textsc{\\ \\ \indent Institute of Mathematics, Jagiellonian University, Łojasiewicza 6, 30-348
	Kraków, Poland. }

\textit{E-mail address:}  \href{mailto:krzysztof.maciaszek@im.uj.edu.pl}{krzysztof.maciaszek@im.uj.edu.pl} 
	\end{document}